\documentclass[12pt]{amsart}
\usepackage{amssymb,amsthm}
\usepackage{xypic}

\theoremstyle{plain}
\newtheorem{lemma}{Lemma}

\newtheorem{cor}{Corollary}

\newtheorem*{thm*}{Theorem}
\theoremstyle{remark}
\newtheorem*{rmk}{Remark}

\newtheorem*{rmks}{Remarks}
\newcommand{\bil}[2]{\langle #1, #2 \rangle}
\newcommand{\Z}{{\mathbb{Z}}}
\newcommand{\N}{{\mathbb{N}}}
\newcommand{\Ext}{\operatorname{Ext}}
\newcommand{\Hom}{\operatorname{Hom}}

\newcommand{\M}{\mathsf{M}}
\newcommand{\GL}{\mathsf{GL}}
\newcommand{\ind}{\operatorname{ind}}
\renewcommand{\le}{\leqslant}
\renewcommand{\ge}{\geqslant}

\newcommand{\B}{\mathbf{B}}
\newcommand{\HH}{\mathbf{H}}

\newcommand{\ch}{\operatorname{ch}}

\newcommand{\mull}{\operatorname{Mull}}
\newcommand{\sgn}{\operatorname{sgn}}

\begin{document}
\title{Hook modules for general linear groups}
\author{Stephen~Doty}\thanks{The first author was supported by the
  inaugural Yip Fellowship at Magdalene College, Cambridge}
\address{Mathematics and Statistics, Loyola University Chicago,
  Chicago, IL 60626, USA} 

\author{Stuart~Martin} \address{Magdalene
  College, Cambridge, CB3 0AG, England, UK}
\date{24 February 2008}
\begin{abstract}
For an arbitrary infinite field $k$ of characteristic $p > 0$, we
describe the structure of a block of the algebraic monoid $\M_n(k)$
(all $n \times n$ matrices over $k$), or, equivalently, a block of the
Schur algebra $S(n,p)$, whose simple modules are indexed by $p$-hook
partitions. The result is known; we give an elementary and
self-contained proof, based only on a result of Peel and Donkin's
description of the blocks of Schur algebras.  The result leads to a
character formula for certain simple $\GL_n(k)$-modules, valid for all
$n$ and all $p$. This character formula is a special case of one found
by Brundan, Kleshchev, and Suprunenko and, independently, by Mathieu
and Papadopoulo. 
\end{abstract}

\maketitle

\section*{Introduction}\noindent
We describe the structure of the family of Weyl modules labeled by
$p$-hooks, for $\GL_n(k)$ where $k$ is an infinite field of
characteristic $p>0$. The main result, given in the Theorem of Section
\ref{results}, leads to a description of the module structure of the
family of projective-injective tilting modules labeled by $p$-hooks,
and determines completely the corresponding block of $\M_n(k)$.  The
result was previously obtained in \cite[\S5.7]{DEN} by a different
argument; see also \cite{Hemmer}.  Another consequence of the Theorem
is a character formula for simple $\GL_n(k)$-modules labeled by
$p$-hooks; this character formula was previously conjectured by
Jantzen. After the first draft of this paper was written, it was
pointed out to us that the character formula is an easy consequence of
a result of Brundan, Kleshchev, and Suprunenko \cite[Theorem 6.3]{BKS}
describing the character of simple modules labeled by `completely
splittable' weights. The same character formula was obtained
independently by Mathieu and Papadopoulo \cite{MP} by a different
method, and it can be used to give yet another proof of the Theorem of
this note.

The idea behind the proof given here is to compare with the symmetric
group $\Sigma_p$ on $p$ letters using the idempotent `Schur functor'
of \cite[\S6]{Green}, given by $M \to eM$ where $e$ is projection onto
the $(1^p)$ weight space.  The key lemma (Lemma \ref{lem:B}) states
that the $p$-hooks index the simple modules in a block of the Schur
algebra $S(n,p)$, for $n \ge p$. This provides a tight upper bound on
the number of composition factors of a Weyl module labeled by a given
$p$-hook; the corresponding lower bound is provided by Jantzen's sum
formula (see Lemma \ref{lem:A}).

We note that the main result also provides a set of examples where
\[
\Ext^1_{\GL_n(k)}(L(\lambda), L(\mu)) \simeq
\Ext^1_{\Sigma_p}(eL(\lambda), eL(\mu)).
\]
The question of determining pairs $\lambda, \mu$ for which such an
equality holds has been studied in \cite{Fettes}, \cite{SM},
\cite{KN}.

\section{Notation}\label{notation}\noindent 
We mostly follow the notational conventions of \cite[\S1]{Green}.  Fix
an infinite field $k$ of positive characteristic $p$, and set $G
= \GL_n(k)$, the general linear group of invertible $n \times n$
matrices over $k$.  Let $E_{ij}$ ($1 \le i,j \le n$) denote the matrix
units in $\M_n(k)$ (the algebraic monoid of $n \times n$ matrices over
$k$). So $E_{ij}$ is the $n \times n$ matrix with $(i,j)$-entry 1 and
all other entries 0.  The $E_{ij}$ form a $k$-basis of $\M_n(k)$; let
$c_{ij}$ ($1 \le i,j \le n$) be the dual basis of $\M_n(k)^*$. The
coordinate algebra $k[G]$ is generated by the $c_{ij}$ and
$\det(c_{ij})^{-1}$.

Let $T$ be the subgroup of $G$ consisting of all diagonal
matrices.  The restrictions $\varepsilon_i = {c_{ii}}_{|_T}$ ($1 \le i
\le n$) form a basis of the character group $X = X(T) = \Hom(T, G_m)$
(algebraic group homomorphisms to the multiplicative group $G_m$); we
identify elements of $X$ with elements of $\Z^n$ via the
correspondence
\[ \textstyle
  \lambda = \sum_i \lambda_i \varepsilon_i \leftrightarrow (\lambda_1,
  \dots, \lambda_n).
\]
The root system associated to the pair $(G,T)$ is the set $R = \{
\varepsilon_i - \varepsilon_j: 1 \le i \ne j \le n \}$. We choose as
positive roots the system $R^+ = \{ \varepsilon_i - \varepsilon_j: 1
\le i < j \le n \}$; then $S = \{ \varepsilon_i - \varepsilon_{i+1} :
1 \le i \le n-1 \}$ is the corresponding set of simple roots.  Let $B$
be the Borel subgroup of $G$ corresponding with the negative
roots. The set of dominant weights is the set
\[
  X^+ = \{\lambda \in X: \bil{\lambda}{\alpha^\vee} \ge 0, \text{ all
  $\alpha \in S$} \};
\]
in terms of our identification $X \simeq \Z^n$ we have that $\lambda =
(\lambda_1, \dots, \lambda_n) \in X^+$ if an only if 
\[
\lambda_1 \ge \lambda_2 \ge \cdots \ge \lambda_n.
\]
The Weyl group $W$ associated to $(G,T)$ is naturally isomorphic to
the symmetric group $\Sigma_n$ on $n$ letters, via the correspondence
$s_{\alpha} \to (i,j)$ when $\alpha = \varepsilon_i - \varepsilon_j$.
As usual, $W$ acts on $X$ via the rule $s_\alpha \lambda = \lambda -
\bil{\lambda}{\alpha^\vee} \alpha$; we shall need the `dot action'
given by $s_\alpha \cdot \lambda = s_\alpha(\lambda+\rho) - \rho$
where $\rho = \sum_i (n-i)\varepsilon_i$.  (Here $\rho$ is not the
usual `half sum of the positive roots,' which is not always defined in
$X$, but it has the crucial property $\bil{\rho}{\beta^\vee} = 1$ for
all $\beta \in S$.)  We also need to consider elements of the affine
Weyl group $W_p$ generated by all $s_{\alpha,ap}$ ($\alpha\in R$,
$a\in \Z$) where $s_{\alpha,ap} = s_\alpha + ap\alpha$.

Let $k_\lambda$ (for $\lambda \in X$) be the one dimensional
$B$-module such that $T$ acts via the character $\lambda$ and the
unipotent radical of $B$ acts trivially.  For $\lambda \in X^+$ we
have the following rational $G$-modules:
\begin{align*}
\nabla(\lambda) & = \ind_B^G k_\lambda \\
\Delta(\lambda) & = {}^\tau \nabla(\lambda) 
                  \simeq \nabla(-w_0\lambda)^* \text{ (the Weyl module)}\\
T(\lambda) &= \text{indecomposable tilting module of 
                    highest weight } \lambda .
\end{align*}
Here $T(\lambda)$ is the unique indecomposable module of highest
weight $\lambda$ which admits both a $\Delta$ and $\nabla$ filtration
(see \cite{Donkin:tilt}), $\tau$ is the `transpose dual' operator
(\cite[Part II, 2.13]{Jantzen}) and $w_0$ is the longest element of
$W$.

Any finite-dimensional $G$-module $M$ is the direct sum of its weight
spaces: $M = \oplus_{\lambda\in X} M_\lambda$ where $M_\lambda = \{ m
\in M: tm = \lambda(t) m, \text{ all } t\in T\}$. We set
\[
\ch M = \textstyle \sum_{\lambda}\, (\dim M_\lambda)\,e(\lambda) \in
\Z[X]
\]
(the formal character of $M$). Here $\Z[X]$, the group ring of $X$, is
the free $\Z$-module with basis $\{ e(\lambda): \lambda \in X \}$ and
multiplication $e(\lambda)e(\mu) = e(\lambda+\mu)$.

A $\GL_n(k)$-module $M$ is termed {\em polynomial} if it lifts to a
rational $\M_n(k)$-module. It is well known (see
e.g.~\cite[\S1]{Green}) that restriction from $\M_n(k)$ to $\GL_n(k)$
induces an full embedding of the category $\M_n(k)$-mod of rational
left $\M_n(k)$-modules in the category $\GL_n(k)$-mod of rational left
$\GL_n(k)$-modules.  The weights of an object $M \in \M_n(k)$-mod are
of the form $\lambda = (\lambda_1, \dots, \lambda_n) \in \N^n$.  Such
weights are often called polynomial weights. In the language of
combinatorics, polynomial weights are compositions (i.e.\ unordered
partitions) of length $n$ and dominant polynomial weights are
partitions of length at most $n$.  We omit any zero parts when writing
a partition, as usual.

The category of polynomial $\GL_n(k)$-modules is graded by homogeneous
degree; that is we have an equivalence 
\[
\M_n(k)\text{-mod} \simeq \oplus_{r\ge 0}\; S(n,r)\text{-mod}
\]
in the sense that every rational $\M_n(k)$-module may be written as a
direct sum of homogeneous ones of various degrees.  Here $S(n,r)$ is
the Schur algebra in degree $r$; its module category $S(n,r)$-mod is
equivalent with the full subcategory of $\M_n(k)$-mod consisting of
homogeneous modules of degree $r$. We write $\lambda \vdash r$ to
indicate that $\lambda$ is a partition of $r$ (i.e., $\sum \lambda_i =
r$); the set of dominant weights for $S(n,r)$-mod is precisely the set
of all $\lambda \vdash r$ such that length $\le n$.

{\em Assume that $n \ge r$ for the remainder of this section.} Then
the set of $\lambda \vdash r$ with length $\le n$ is the same as the
set of partitions of $r$, and we have a so-called `Schur functor' from
$S(n,r)$-mod to $k\Sigma_r$-mod, where $\Sigma_r$ is the symmetric
group on $r$ letters. This is an exact covariant functor given by the
rule $M \mapsto eM$ where $e \in S(n,r)$ is the idempotent projector
onto the $(1^r)$-weight space.  See \cite[Chapter 6]{Green} or
\cite[Chapter 4]{Martin} for details. In particular we have
\begin{equation} \label{f-effect-std}
e\Delta(\lambda) \simeq S^{\lambda}, \qquad 
e\nabla(\lambda) \simeq S_{\lambda}
\end{equation}
for any partition $\lambda \vdash r$, where $S^\lambda$ (resp.,
$S_{\lambda}$) is the Specht (resp., dual Specht) module indexed by
$\lambda$.  Here $S^\lambda$ may be defined (following
\cite[\S4]{James}) as the submodule of $M^\lambda$ spanned by
polytabloids of type $\lambda$, where $M^\lambda$ is the transitive
permutation module $\ind_{\Sigma_\lambda}^{\Sigma_r} k$ associated
with the Young subgroup $\Sigma_\lambda$ corresponding with $\lambda$,
and $S_\lambda \simeq (S^\lambda)^*$ as $k\Sigma_r$-modules.

A partition $\lambda \vdash r$ is $p$-regular if there does not exist
an $i$ such that $\lambda_{i+1} = \cdots = \lambda_{i+p}$ and column
$p$-regular if its conjugate $\lambda'$ is $p$-regular. Equivalently,
$\lambda$ is column $p$-regular if and only if $\lambda_i -
\lambda_{i+1} < p$ for all $i$. If $\lambda$ is $p$-regular then
$S^\lambda$ has a unique top composition factor $D^\lambda$; similarly
if $\lambda$ is column $p$-regular then $S_\lambda$ has a unique top
composition factor $D_\lambda$. The set
\[
  \{ D^\lambda: \lambda \vdash r, \lambda \text{ $p$-regular} \} \simeq
  \{ D_\lambda: \lambda \vdash r, \lambda \text{ column $p$-regular} \}
\]
gives a complete set of isomorphism classes of simple
$k\Sigma_r$-modules. Since we have an isomorphism $S^{\lambda'} \simeq
S_{\lambda} \otimes \sgn$ it follows that the two labellings
$\{D^\lambda\}$, $\{D_\lambda\}$ are related by
\begin{equation} \label{sgn-on-D}
D^{\lambda'} \simeq D_\lambda \otimes \sgn
\end{equation}
for any column $p$ regular $\lambda \vdash r$.  We also have
\begin{equation}\label{f-effect-L}
  eL(\lambda) \simeq D_\lambda
\end{equation}
for all column $p$-regular partitions $\lambda \vdash r$. Finally,
since tensoring by the one dimensional sign representation must take
simples to simples, we have $D^\lambda \otimes \sgn \simeq
D^{\mull(\lambda)}$ for all $p$-regular $\lambda \vdash r$, where
$\mull(\lambda)$ is given by a combinatorial procedure described in
\cite{Mull} and proved in \cite{Ford-Kleshchev}.

\section{Results}\label{results}\noindent 
Let $\lambda^i = (p-i, 1^i) = (p-i)\varepsilon_1 + \sum_{j=2}^{i+1}
\varepsilon_j$ for $0 \le i \le p-1$. These are the $p$-hook
partitions.

\begin{lemma} \label{lem:A}
  For $0\le i < \min(n-1,p-1)$ the Weyl module $\Delta(\lambda^i)$ has
  at least two composition factors, of highest weight $\lambda^i$ and
  $\lambda^{i+1}$.
\end{lemma}

\begin{proof}
(We are grateful to Jens Jantzen for sending us this argument).  It is
enough to show that $L(\lambda^{i+1})$ is a composition factor of
$\Delta(\lambda^i)$ for $i < \min(n-1,p-1)$.  Since $\lambda^i -
\lambda^{i+1}$ is equal to the sum of the first $i$ simple roots, one
can reduce to the root system of type $A_i$. In that case Jantzen's
sum formula \cite[Part II, 8.19]{Jantzen} has just one term, namely
the Weyl character of $\lambda^{i+1}$. The result follows.
\end{proof}

Let $r$ be an arbitrary natural number. Given a partition $\lambda
\vdash r$ of length $\le n$ let us denote by $d(\lambda)$ the maximum
of all $d\ge 0$ such that $\lambda_i - \lambda_{i+1} \equiv -1$ modulo
$p^d$ for all $1 \le i < n$.  We need S.~Donkin's result from
\cite{Donkin:SA4}, which states that for partitions $\lambda, \mu
\vdash r$ of length $\le n$, the corresponding simple modules
$L(\lambda)$, $L(\mu)$ lie in the same block for $S(n,r)$ if and only
if both conditions (B1) and (B2) below hold:
\begin{enumerate}
\item[(B1)] $d(\lambda) = d(\mu)$ (say $d = d(\lambda) = d(\mu)$); 

\item[(B2)] there exists $w \in W$ such that $\lambda_i - i \equiv
\mu_{w(i)} - w(i)$ (mod $p^{d+1}$) for all $1 \le i \le n$.
\end{enumerate}
Noting that $\lambda$ and $\mu$ will satisfy (B2) if and only if
$\lambda+(n^n)$ and $\mu+(n^n)$ also satisfy (B2), we see that (B2) is
equivalent to the condition
\begin{enumerate}
\item[(B$2'$)]  there exists $w \in W$ such that $\lambda + \rho \equiv
w(\mu+\rho)$ (mod $p^{d+1}$).
\end{enumerate}
Here, for $n$-part compositions $\lambda' = (\lambda'_1, \dots,
\lambda'_n)$, $\lambda'' = (\lambda''_1, \dots, \lambda''_n)$ we
declare that $\lambda' \equiv \lambda''$ (mod $N$) if and only if
$\lambda'_i \equiv \lambda''_i$ (mod $N$) for all $1 \le i \le n$.

\begin{lemma} \label{lem:B}
  Assume that $n \ge p$. We identify a block with the set of highest
  weights labeling its simple modules. With that identification, the
  block of the Schur algebra $S(n,p)$ containing the one row partition
  $\lambda^0 = (p)$ consists of all the $p$-hook partitions, and only
  those partitions.
\end{lemma}

\begin{proof}
  We may assume without loss of generality that $n=p$ since the block
  is the same for larger $n$.  We have $d = d(\lambda^i) = 0$ for all
  $0 \le i \le p-1$. Set $\lambda = \lambda^0$. The modulo $p$
  residues of the parts of $\lambda+\rho$ in order are $p-1, p-2,
  \dots, 1, 0$. Thus, in order that a partition $\mu \vdash p$ satisfy
  condition (B$2'$) in relation to $\lambda$, it is necessary and
  sufficient that the modulo $p$ residues of $\mu+\rho$ are pairwise
  distinct.

  It is easy to check that this condition holds true for all the
  $p$-hook partitions $\lambda^i$. In fact, one checks by direct
  calculation that $\lambda^i + \rho \equiv s_\alpha(\lambda^{i+1} +
  \rho)$ for all $0 \le i < p-1$, where $\alpha = \varepsilon_1 -
  \varepsilon_{i+1}$. This shows that the block in question contains
  at least all the $p$-hooks. 

  To finish, we need to show that it contains no other
  partition. Suppose that $\mu \vdash p$ is not a
  $p$-hook. Equivalently, $\mu_2 \ge 2$. This forces $\mu_1 \ge 2$ and
  $\mu_{p-1} = \mu_p = 0$ as well. If $\mu_2 = 2$ then $\mu$ cannot
  satisfy the criterion in the first paragraph since the modulo $p$
  residue of $\mu+\rho$ in place 2 is zero and matches the modulo $p$
  residue of $\mu+\rho$ in place $p$. If $\mu_2 =3$ then the last
  three parts of $\mu$ must be zero and  the modulo $p$
  residue of $\mu+\rho$ in place 2 is $1$ and matches the modulo $p$
  residue of $\mu+\rho$ in place $p-1$. The argument continues in this
  way.
\end{proof}

\begin{thm*} 
  $\Delta(\lambda^i)$ has two composition factors $L(\lambda^i)$,
  $L(\lambda^{i+1})$ for all $0\le i < \min(n-1,p-1)$. If $i =
  \min(n-1,p-1)$ then $\Delta(\lambda^i)$ is simple.
\end{thm*}

\begin{proof} First we assume that $n \ge p$. 
  We apply the Schur functor to $\Delta(\lambda^i)$.  By
  \eqref{f-effect-std} we have $e\Delta(\lambda^i) \simeq
  S_{\lambda^i}$. By a theorem of Peel \cite{Peel} (see also
  \cite[Theorem~24.1]{James}), for $p>2$ the dual Specht modules
  labeled by $p$-hooks $\lambda^i$ have at most two composition
  factors. More precisely, $S_{\lambda^i}$ has exactly two composition
  factors if $0<i<p-1$ and just one if $i=0$ or $p-1$.

  By Lemma \ref{lem:B} and the isomorphism \eqref{f-effect-L}, it
  follows that for $0<i<p-1$ no composition factor of
  $\Delta(\lambda^i)$ is killed by the Schur functor. Thus it follows
  from Peel's result that the $\Delta(\lambda^i)$ must have at most
  two composition factors when $p>2$. Combining this with Lemma
  \ref{lem:A}, the Theorem follows for all $0 < i \le p-1$ in case
  $p>2$.

  In case $p=2$ or $i=0$ the result is easy to prove directly. For
  instance, one may apply the main result of \cite{Doty} to
  $\nabla(\lambda^0)$ since this is isomorphic with a symmetric power
  of the natural module.

  Now that the result has been established in case $n \ge p$, we
  consider the case $n<p$. In this case there is an idempotent Schur
  functor sending $S(p,p)$-mod to $S(n,p)$-mod. By the results in
  \cite[\S6.5]{Green} we obtain the result in general. 
\end{proof}

\begin{rmk}
  The Theorem also follows from the character formula of
  \cite{BKS}. Take $0 \le i \le \min(n,p-1)$ as in the Theorem. The
  dimension of the $\mu$ weight space of $\Delta(\lambda^i)$ is equal
  to the number of standard $\lambda^i$-tableaux of weight $\mu$. The
  latter set is the disjoint union of the set of standard
  $\lambda^i$-tableaux with no `bad' $p$-hook (in the sense of \S6 of
  \cite{BKS}) and the image of the set of standard
  $\lambda^{i+1}$-tableaux of weight $\mu$ with no bad $p$-hook under
  the map sending a tableau of shape $\lambda^{i+1}$ to one of shape
  $\lambda^i$ by moving its bottom node up to the end of its top
  row. This shows that $\ch \Delta(\lambda^i) = \ch L(\lambda^i) + \ch
  L(\lambda^{i+1})$, as desired.
\end{rmk}

We now give some consequences of the Theorem.

\begin{cor} \label{cor:A}
  We have $\ch L(\lambda^i) = \sum_{j \ge i} (-1)^{j-i} \ch
  \Delta(\lambda^j)$ for all $0\le i \le \min(n-1,p-1)$.
\end{cor}

\begin{proof}
This follows immediately from the Theorem, by induction on $i$
starting with the base case $i = \min(n-1,p-1)$ and working backwards
in $i$.
\end{proof}

\begin{cor}\label{cor:B}
For all $\mu \vdash p$ with at most $n$ parts and all $0\le i \le
\min(n-1,p-1)$ we have:
\[
\Ext^1_{\GL_n(k)}(L(\lambda^i), L(\mu)) \simeq k
\] 
in case: 
\begin{enumerate}
\item[(i)] $0<i<\min(n-1,p-1)$ and $\mu = \lambda^{i+1}$ or
$\lambda^{i-1}$;

\item[(ii)] $i=0$ and $\mu = \lambda^1$;

\item[(iii)] $i=\min(n-1,p-1)$ and $\mu = \lambda^{i-1}$. 
\end{enumerate}
For all other cases $\Ext^1_{\GL_n(k)}(L(\lambda^i), L(\mu)) = 0$.
\end{cor}

\begin{proof}
It is known that for partitions $\lambda \vdash r$, $\mu \vdash r$ we
have isomorphisms (for any $r$)
\[
\Ext^1_{\GL_n(k)}(L(\lambda), L(\mu)) \simeq 
\Ext^1_{\M_n(k)}(L(\lambda), L(\mu)) \simeq 
\Ext^1_{S(n,r)}(L(\lambda), L(\mu)) .
\]
The result now follows from the Theorem by Lemma 2 and \cite[Part II,
  (2.14)(4)]{Jantzen}).
\end{proof}

Given a partition $\lambda$ of length at most $n$ we denote by
$P(\lambda)$ (respectively, $I(\lambda)$) the projective hull
(respectively, injective envelope) of $L(\lambda)$ in the category
$\M_n(k)$-mod.  We have $P(\lambda) \simeq {}^\tau I(\lambda)$.
Recalling the Mullineux map \cite{Mull}, by \cite[Lemma 3.3]{DeDo} we
have that
\begin{equation} \label{Mull}
I(\lambda) = T(\mull(\lambda')) \text{ if $\lambda$ is column
$p$-regular.}
\end{equation}
Moreover, in that case $\mull(\lambda')$ is again a partition of
length at most $n$. We recall from \cite[(2.2h)]{Donkin:SA1}
that $I(\lambda)$ has a $\nabla$-filtration and for all partitions
$\lambda, \mu$ of length $\le n$ the number $(I(\lambda):
\nabla(\mu))$ of subquotients in the filtration isomorphic with
$\nabla(\lambda)$ satisfies the reciprocity law
\begin{equation} \label{reciprocity}
  (I(\lambda): \nabla(\mu)) = [\nabla(\mu): L(\lambda)]
\end{equation}
where the number on the right-hand side stands for the composition
factor multiplicity of $L(\lambda)$ in a composition series of
$\nabla(\mu)$. Another corollary of our main result is the following.

\begin{cor}\label{cor:C}
  Suppose that $p>2$.  The module structure of $P(\lambda^i) =
  I(\lambda^i) = T(\lambda^{i-1})$ is as follows, for all $1 \le i \le
  \min(n-1,p-1)$ (the left diagram is for the case $i < \min(n-1,p-1)$
  and the right one for $i = \min(n-1,p-1)$):
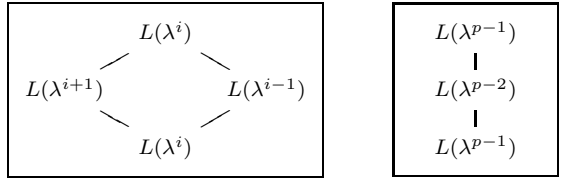
\[
\begin{tabular}{cc}
\framebox[42mm]{
\begin{minipage}{42mm}
\def\objectstyle{\scriptstyle}
\xymatrix@=6pt{
  & L(\lambda^i) \ar@{-}[dl] \ar@{-}[dr] & \\
L(\lambda^{i+1}) \ar@{-}[dr] & & L(\lambda^{i-1}) \ar@{-}[dl] \\
  & L(\lambda^i)  &
}\end{minipage}
}
\quad
&
\quad
\framebox[22mm]{
\begin{minipage}{22mm}
\def\objectstyle{\scriptstyle}
\xymatrix@=6pt{
  & L(\lambda^{p-1}) \ar@{-}[d] & \\
  & L(\lambda^{p-2}) \ar@{-}[d] & \\
  & L(\lambda^{p-1}) &
}\end{minipage}
}
\end{tabular}
\]
where the module diagram is interpreted as described in \cite{alp}.
For $p=2$ the module structure of $P(\lambda^1) = I(\lambda^1) =
T(\lambda^{0})$ is depicted in the right diagram above.
\end{cor}

\begin{proof}
From Lemma \ref{lem:B}, the reciprocity law \eqref{reciprocity}, and
the Theorem it follows that for each $i$, $1 \le i \le p-1$, the
module $I(\lambda^i)$ has a $\nabla$-filtration with subquotients
$\nabla(\lambda^i)$ and $\nabla(\lambda^{i+1})$, each occurring with
multiplicity one. From this and Corollary \ref{cor:B} it follows that
the module structure of $I(\lambda^i)$ must be as described in all
cases.

Since the module $I(\lambda^i)$ is contravariantly self-dual, i.e.,
$I(\lambda^i) \simeq {}^\tau I(\lambda^i)$, it follows immediately
that $I(\lambda^i) = P(\lambda^i)$. Hence it follows that
$I(\lambda^i) = P(\lambda^i)$ is a tilting module. Its highest weight
is $\lambda^{i-1}$, so we conclude that $I(\lambda^i) = P(\lambda^i) =
T(\lambda^{i-1})$. This completes the proof.
\end{proof}

\begin{rmks}
1. By the equality \eqref{Mull} it follows from the last corollary
  that $\mull((\lambda^i)') = \lambda^{i-1}$. This may also be checked
  combinatorially.

2. From \eqref{reciprocity} it also follows that $I(\lambda^0) \simeq
   \nabla(\lambda^0)$ and thus that $P(\lambda^0) \simeq
   \Delta(\lambda^0)$. These modules are not tilting. 
\end{rmks}

\section{Comparison with symmetric groups}\label{sec:symm}\noindent
It will be illuminating to compare our results with well known results
concerning blocks of cyclic defect group for symmetric groups.  We are
concerned with the group $\Sigma_p$ in characteristic $p$, which has
just one block $\B=k\Sigma_p$.  

For convenience we assume first that $p>2$.  We label the simple
$\B$-modules by the column $p$-regular partitions $\lambda^i$ for
$1\le i \le p-1$ and denote them by $D_i=D_{\lambda^i}$.  (The reader
who prefers to label by $p$-regular partitions should use the
isomorphism $D_{\lambda^i} \simeq D^{\lambda^{i-1}}$.)  As follows
from \cite{Peel}, the Brauer tree of $\B$ is an open polygon with no
exceptional vertex, as depicted in the figure below.
\[
\def\objectstyle{\scriptstyle}
\xymatrix@R=12pt{
  {\bullet} \ar@{-}[r]^{D_1} & 
  {\bullet} \ar@{-}[r]^{D_2} & 
  {\bullet} \ar@{.}[r] & 
  {\bullet} \ar@{-}[r]^{D_{p-1}} & {\bullet}
}
\]
The edges of the tree are in one-one correspondence with the simple
$\B$-modules; the simple modules appearing in order as
$D_i=D_{\lambda^i}$ for $1\le i \le p-1$. (note $D_{p-1}=D_{(1^p)}$),
see \cite[\S2.2]{SM}.  Let $P_i$ be the projective cover of $D_i$.
Then the first diagram below gives the module structure of $P_1$, the
second is $P_i$ (for all $2 \le i \le p-2$) and the third is
$P_{p-1}$:
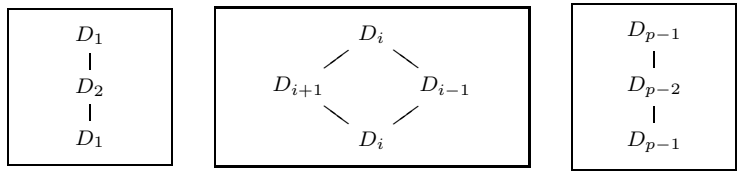
\[
\begin{tabular}{cc}
\framebox[22mm]{
\begin{minipage}{22mm}
\def\objectstyle{\scriptstyle}
\xymatrix@=6pt{
  & D_1 \ar@{-}[d] & \\
  & D_2 \ar@{-}[d] & \\
  & D_1 &
}\end{minipage}
}
\quad
\framebox[42mm]{
\begin{minipage}{42mm}
\def\objectstyle{\scriptstyle}
\xymatrix@=6pt{
  & D_i \ar@{-}[dl] \ar@{-}[dr] & \\
D_{i+1} \ar@{-}[dr] & & D_{i-1} \ar@{-}[dl] \\
  & D_i  &
}\end{minipage}
}
\quad
\framebox[22mm]{
\begin{minipage}{22mm}
\def\objectstyle{\scriptstyle}
\xymatrix@=6pt{
  & D_{p-1} \ar@{-}[d] & \\
  & D_{p-2} \ar@{-}[d] & \\
  & D_{p-1} &
}\end{minipage}
}
\end{tabular}
\]
where again the diagrams are to be interpreted as in \cite{alp}.  

For $p=2$ there is just one simple (namely, the trivial module) and
its projective cover is uniserial of length 2.

Assume that $n \ge p$. Let us denote by $\HH$ the block of $\M_n(k)$
containing the simple modules labeled by the $p$-hooks $\lambda^i$
for $0\le i \le p-1$.  Comparing the results in Corollary \ref{cor:C}
with the description of the block $\B$ in this section, we see that
$\HH$ and $\B$ are nearly equivalent, in the sense that $\HH$ contains
one more simple (and thus one more projective) than does $\B$ and upon
deleting all references to the offending simple module (and its
projective) from $\HH$ we recover $\B$. This deletion procedure is
precisely the effect of the Schur functor applied to $\HH$.


\end{document}